\documentclass[12pt]{article}

\usepackage[T1]{fontenc}
\usepackage[utf8]{inputenc}
\usepackage[english]{babel}

\usepackage{amsmath}
\usepackage{amssymb}
\usepackage{amsthm}

\usepackage[a4paper,margin=2.5cm]{geometry}
\usepackage{microtype}
\usepackage[hidelinks]{hyperref}

\hypersetup{
  pdftitle={Dedekind Poisson Algebras over Arbitrary Fields},
  pdfauthor={A. I. Plakosh; O. O. Pypka},
  pdfsubject={Poisson algebras whose subalgebras are ideals},
  pdfkeywords={Poisson algebra, Dedekind algebra, subalgebra, ideal, bilinear form, arbitrary field, characteristic 2}
}

\usepackage{authblk}

\newtheorem{theorem}{Theorem}[section]
\newtheorem{proposition}[theorem]{Proposition}
\newtheorem{lemma}[theorem]{Lemma}
\newtheorem{corollary}[theorem]{Corollary}

\theoremstyle{definition}
\newtheorem{definition}[theorem]{Definition}
\newtheorem{example}[theorem]{Example}

\theoremstyle{remark}
\newtheorem{remark}[theorem]{Remark}

\title{Dedekind Poisson Algebras over Arbitrary Fields}

\author[1]{A.~I.~Plakosh}
\author[2]{O.~O.~Pypka}

\affil[1]{National Academy of Sciences in Ukraine, Institute of Mathematics,
Kyiv, Ukraine}
\affil[2]{Oles Honchar Dnipro National University,
Dnipro, Ukraine}

\date{}

\begin{document}

\maketitle

\begin{abstract}
We study Poisson algebras in which every Poisson subalgebra is a Poisson ideal, called Dedekind Poisson algebras, over arbitrary fields and without any finite-dimensionality assumption. In characteristic different from \(2\), combining the two operations by $x*y=xy+[x,y]$ connects the problem with Outcalt’s classical classification of power-associative \(H\)-algebras and yields the same structural type. We obtain a complete classification by a direct two-operation argument valid in every characteristic. This approach also shows that characteristic \(2\) retains additional Poisson information which cannot be recovered from the single product \(*\).

\end{abstract}

\medskip


\medskip


\medskip

\section{Introduction}
\label{sec:introduction}

Conditions requiring every subobject to be normal or ideal are among
the classical rigidity conditions of algebra. Dedekind groups, in
which every subgroup is normal, originate in Dedekind's work
\cite{Dedekind1897}; the non-abelian case was further developed by
Baer \cite{Baer1933}. Analogous questions were later studied for
rings and algebras. Liu studied alternative and Jordan algebras in
which every subalgebra is an ideal \cite{Liu1964}, Kruse considered
the corresponding condition for rings \cite{Kruse1968}, and Outcalt
subsequently obtained a classification of power-associative algebras
over fields of characteristic different from $2$ in which every
subalgebra is a two-sided ideal \cite{Outcalt1967}. A close
non-associative analogue appears in the work of Kurdachenko, Semko
and Subbotin on Leibniz algebras whose every subalgebra is an ideal
\cite{KurdachenkoSemkoSubbotin2017}.

Poisson algebras carry two operations on the same vector space: a
commutative associative multiplication and an alternating Lie
multiplication tied by the Leibniz identity.  Their structural theory
has been developed in several directions; see, for example,
\cite{FernandezNavarroTowers2024,FernandezOmirov2025,GozeRemm2008,
KurdachenkoPypkaSubbotin2021,SicilianoUsefi2021,Towers2026}. The
presence of two interacting operations makes the condition
\[
\text{every Poisson subalgebra of $P$ is a Poisson ideal of $P$}
\]
a natural Poisson version of the Dedekind problem.  We call such an
algebra a \emph{Dedekind Poisson algebra}.

There is an important classical connection which must be separated
from the genuinely two-operation part of the problem.  Suppose that
$\operatorname{char}(F)\neq2$ and define one bilinear product by
\[
x*y=xy+[x,y].
\]
Then the two Poisson operations are recovered from $*$ by
\[
xy=\frac{x*y+y*x}{2},
\qquad
[x,y]=\frac{x*y-y*x}{2}.
\]
Consequently, the $*$-subalgebras are exactly the Poisson subalgebras
and the two-sided $*$-ideals are exactly the Poisson ideals.  Moreover,
the algebra $(P,*)$ is power-associative, since the bracket vanishes on
every one-generated Poisson subalgebra.  Thus, in characteristic
different from $2$, Outcalt's theorem \cite{Outcalt1967} applies to
Dedekind Poisson algebras.

The structural form obtained from Outcalt agrees exactly with the
canonical form proved below.  Indeed, if the non-zero nil component of
the Outcalt algebra is written as
\[
u*v=\alpha(u,v)c,
\]
then
\[
\beta(u,v)=\frac{\alpha(u,v)+\alpha(v,u)}2,
\qquad
\omega(u,v)=\frac{\alpha(u,v)-\alpha(v,u)}2
\]
recover respectively the associative and Lie parts, and
$\alpha(v,v)=\beta(v,v)$.  Thus our classification in
$\operatorname{char}(F)\neq2$ is consistent with the classical
$H$-algebra classification, although the proof given here is obtained
directly from the two Poisson operations and does not use Outcalt's
theorem.

Characteristic $2$ is essentially different.  The single operation
$x*y=xy+[x,y]$ no longer determines the associative multiplication and
the Lie bracket separately.  In fact, different pairs of Poisson
operations can give the same product $*$.  Therefore the
characteristic-$2$ problem is not recovered from the classical
one-operation classification.  The main purpose of the present paper
is to give a uniform direct proof over arbitrary fields, retain the two
operations throughout, and describe explicitly the additional
characteristic-$2$ information.  To the best of our knowledge, a
complete two-operation classification of this condition in
characteristic $2$ has not previously been given.

Our starting point is the simple structure of one-generated Poisson
algebras:
\[
\langle a\rangle_P
 =\operatorname{span}_F\{a,a^2,a^3,\ldots\},
\qquad
[\langle a\rangle_P,\langle a\rangle_P]=0.
\]
We prove that the Dedekind condition is equivalent to requiring every
such cyclic subalgebra to be a Poisson ideal.  This immediately yields
\[
[P,P]\subseteq P^2
\]
and implies that every associative subalgebra of $P(+,\cdot)$ is in
fact a Poisson ideal.  We then give an elementary analysis of the
associated commutative associative algebra which is valid in arbitrary
characteristic and does not assume that $P$ is finite-dimensional.
It gives
\[
P(+,\cdot)=E\oplus N,
\]
where $E=0$ or $E=Fe$ with $e^2=e$ and $eN=0$, while
\[
N^3=0,
\qquad
\dim_F N^2\leq1,
\qquad
\operatorname{Ann}(N)=\{x\in N:x^2=0\}.
\]

The Lie multiplication is forced into the same one-dimensional part:
\[
[P,P]=[N,N]\subseteq N^2.
\]
If $N^2=Fc\neq0$, then $Fc$ and the whole associative annihilator of
$N$ are Lie-central.  Choosing vector-space complements gives a direct
sum of vector spaces
\[
P=E\oplus Z\oplus V\oplus Fc,
\]
where $c\neq0$ and
\[
uv=\beta(u,v)c,
\qquad
[u,v]=\omega(u,v)c
\qquad(u,v\in V).
\]
Here $\beta$ is symmetric and satisfies
$\beta(v,v)\neq0$ for every $v\neq0$, while $\omega$ is arbitrary
alternating.  No additional compatibility relation between the two
forms occurs.  We also determine the isomorphism relation on these
data.

The characteristic-$2$ consequences make the two-operation nature of
the result particularly visible.  Writing
\[
F^2=\{a^2:a\in F\},
\]
we prove for the active component $V$ that
\[
\dim_FV\leq [F:F^2].
\]
Hence over a perfect field of characteristic $2$ one has
$\dim_FV\leq1$ and the Lie multiplication is zero, whereas imperfect
fields admit genuinely mixed examples.  More precisely, mixed
Dedekind Poisson algebras exist in characteristic $2$ exactly over
imperfect fields.  Outside characteristic $2$, the corresponding
criterion is the existence of a non-square, in agreement with the
classical Outcalt picture.

The paper is organized as follows.  Section~\ref{sec:preliminaries}
contains the basic definitions and the cyclic criterion.
Section~\ref{sec:associative} explains the relation with Outcalt outside
characteristic $2$ and begins the direct associative analysis.
Section~\ref{sec:nilpart} determines the nil part.  In
Section~\ref{sec:lie} we analyze the Lie multiplication and obtain the
canonical pair of bilinear forms.  The classification and its converse
are proved in Section~\ref{sec:classification}; isomorphisms are treated
in Section~\ref{sec:isomorphism}.  Finally, Section~\ref{sec:fields}
focuses on ground-field effects, especially characteristic $2$.

\section{Preliminaries and the cyclic criterion}
\label{sec:preliminaries}

Throughout the paper, $F$ denotes an arbitrary field. Unless explicitly
stated otherwise, no finite-dimensionality assumption is made and no
restriction is imposed on $\operatorname{char}(F)$.

A \emph{Poisson algebra} over $F$ is an $F$-vector space $P$ endowed
with a commutative associative multiplication
\[
(x,y)\longmapsto xy
\]
and a Lie multiplication
\[
(x,y)\longmapsto[x,y]
\]
satisfying the Leibniz identity
\[
[xy,z]=x[y,z]+y[x,z]
\]
for all $x,y,z\in P$. The Lie multiplication is understood to be
\emph{alternating}, that is,
\[
[x,x]=0
\qquad(x\in P).
\]
Thus in characteristic $2$ we use alternation, not merely
skew-symmetry, as part of the definition. The associative multiplication
is not assumed to have an identity.

We write
\[
P(+,\cdot)
\qquad\text{and}\qquad
P(+,[\, ,\,])
\]
for the associated commutative associative and Lie algebras,
respectively.

A vector subspace $H\leq P$ is a \emph{Poisson subalgebra} if
\[
HH\subseteq H
\qquad\text{and}\qquad
[H,H]\subseteq H.
\]
A vector subspace $I\leq P$ is a \emph{Poisson ideal} if
\[
PI\subseteq I
\qquad\text{and}\qquad
[P,I]\subseteq I.
\]
All subalgebras are understood in the non-unital sense. In particular,
a subalgebra of a unital Poisson algebra need not contain the identity.

For a subspace $H\leq P$, we denote by $H^2$ the linear span of all
products $xy$ with $x,y\in H$. We put $P^1=P$ and define the
\emph{associative powers}
\[
P^{n+1}=P^nP
\qquad(n\geq1).
\]
The associative algebra $P(+,\cdot)$ is called \emph{nilpotent} if
$P^n=0$ for some $n$. An element $x$ is \emph{nilpotent} if $x^n=0$
for some positive integer $n$, and an associative algebra is called a
\emph{nil algebra} if every one of its elements is nilpotent. These two
notions will be distinguished throughout.

The \emph{associative annihilator} of an associative algebra $A$ is
\[
\operatorname{Ann}(A)=\{a\in A\mid aA=0\}.
\]
When $A=P(+,\cdot)$ we also write $\operatorname{Ann}(P(+,\cdot))$.
The \emph{derived algebra} and the \emph{center} of the associated Lie
algebra are
\[
[P,P]=\operatorname{span}_F\{[x,y]\mid x,y\in P\}
\]
and
\[
\zeta(P(+,[\, ,\,]))
 =\{z\in P\mid[z,P]=0\},
\]
respectively. The associated Lie algebra is \emph{nilpotent of class at
most $2$} when
\[
[[P,P],P]=0.
\]

An element $e$ of an associative algebra is an \emph{idempotent} if
$e^2=e$. An algebra is \emph{unital} if it has an identity element.
A \emph{zero Poisson algebra} is a vector space on which both the
associative multiplication and the Lie multiplication are identically
zero.  When only one multiplication is under discussion, a \emph{zero
algebra} means that this multiplication is identically zero.

For a subset $X\subseteq P$, the Poisson subalgebra generated by $X$
is denoted by $\langle X\rangle_P$. For one element we write
$\langle a\rangle_P$ and call it a \emph{cyclic} or
\emph{one-generated Poisson subalgebra}.

\begin{definition}
\label{def:dedekind}
A Poisson algebra $P$ is called a \emph{Dedekind Poisson algebra} if
every Poisson subalgebra of $P$ is a Poisson ideal of $P$.
\end{definition}

We begin with two elementary ideal properties that will be used
repeatedly.

\begin{lemma}
\label{lem:p2-ann}
For every Poisson algebra $P$, the subspaces $P^2$ and
$\operatorname{Ann}(P(+,\cdot))$ are Poisson ideals.
\end{lemma}

\begin{proof}
Clearly $PP^2\subseteq P^2$. For $x,y,z\in P$, the Leibniz identity
gives
\[
[x,yz]=y[x,z]+z[x,y]\in P^2,
\]
so $[P,P^2]\subseteq P^2$.

Let $a\in\operatorname{Ann}(P(+,\cdot))$. For arbitrary $x,y\in P$,
we have $ay=0$, and hence
\[
0=[ay,x]=a[y,x]+y[a,x].
\]
The first term is zero because $aP=0$. Therefore
$y[a,x]=0$ for every $y\in P$, and thus
$[a,x]\in\operatorname{Ann}(P(+,\cdot))$.
\end{proof}

The next proposition describes one-generated Poisson algebras in the
form needed below.

\begin{proposition}
\label{prop:one-generated}
Let $a\in P$. Then
\[
\langle a\rangle_P
 =\operatorname{span}_F\{a,a^2,a^3,\ldots\}
\]
and
\[
[\langle a\rangle_P,\langle a\rangle_P]=0.
\]
\end{proposition}

\begin{proof}
Put
\[
A=\operatorname{span}_F\{a,a^2,a^3,\ldots\}.
\]
It is an associative subalgebra containing $a$. Since $[a,a]=0$, an
induction using the Leibniz identity gives
\[
[a^m,a]=0
\qquad(m\geq1).
\]
A second induction gives
\[
[a^m,a^n]=0
\qquad(m,n\geq1).
\]
Hence $[A,A]=0$, so $A$ is a Poisson subalgebra containing $a$ and
$\langle a\rangle_P\subseteq A$. The reverse inclusion follows because
all positive powers of $a$ belong to $\langle a\rangle_P$.
\end{proof}

The Dedekind condition is therefore controlled completely by cyclic
subalgebras.

\begin{lemma}[Cyclic criterion]
\label{lem:cyclic-criterion}
A Poisson algebra $P$ is Dedekind if and only if
$\langle a\rangle_P$ is a Poisson ideal of $P$ for every $a\in P$.
\end{lemma}

\begin{proof}
Only the converse needs proof. Let every cyclic Poisson subalgebra be
an ideal and let $H\leq P$ be a Poisson subalgebra. Since
$\langle h\rangle_P\subseteq H$ for every $h\in H$ and every element of
$H$ belongs to its own cyclic subalgebra,
\[
H=\sum_{h\in H}\langle h\rangle_P.
\]
A sum of Poisson ideals is again a Poisson ideal. Hence $H$ is an ideal
of $P$.
\end{proof}

\begin{lemma}
\label{lem:quotients}
Let $P$ be a Dedekind Poisson algebra and let $I$ be a Poisson ideal.
Then $P/I$ is a Dedekind Poisson algebra.
\end{lemma}

\begin{proof}
Let $K/I$ be a Poisson subalgebra of $P/I$. Its inverse image $K$ is a
Poisson subalgebra of $P$ containing $I$. Since $P$ is Dedekind, $K$ is
a Poisson ideal. Therefore $K/I$ is a Poisson ideal of $P/I$.
\end{proof}

The first interaction between the two multiplications is especially
strong.

\begin{proposition}
\label{prop:derived-in-square}
If $P$ is a Dedekind Poisson algebra, then
\[
[P,P]\subseteq P^2.
\]
\end{proposition}

\begin{proof}
By Lemma~\ref{lem:p2-ann}, $P^2$ is a Poisson ideal, and by
Lemma~\ref{lem:quotients}, $P/P^2$ is Dedekind. Its associative
multiplication is zero. Consequently every Lie subalgebra of
$P(+,[\, ,\,])/P^2$ is a Poisson subalgebra and hence a Lie ideal.

We recall the elementary fact that a Lie algebra all of whose
subalgebras are ideals is abelian. Indeed, for $x,y$ linearly
independent, the one-dimensional subalgebras $Fx$ and $Fy$ are ideals,
so
\[
[x,y]\in Fx\cap Fy=0;
\]
for linearly dependent $x,y$ the bracket is zero by alternation. Thus
the associated Lie algebra of $P/P^2$ is abelian, which is equivalent
to $[P,P]\subseteq P^2$.
\end{proof}

\begin{corollary}
\label{cor:p2zero}
If $P$ is Dedekind and $P^2=0$, then $[P,P]=0$.
\end{corollary}

\section{The associated commutative associative algebra}
\label{sec:associative}

We now show that the Dedekind condition forces every associative
subalgebra of $P(+,\cdot)$ to be an ideal.

\begin{proposition}
\label{prop:assoc-subalgebra-ideal}
Let $P$ be a Dedekind Poisson algebra and let $B$ be an associative
subalgebra of $P(+,\cdot)$. Then $B$ is a Poisson subalgebra of $P$ and
an associative ideal of $P(+,\cdot)$.
\end{proposition}

\begin{proof}
Let $x,y\in B$. By the cyclic criterion,
$\langle x\rangle_P$ is a Poisson ideal, so
\[
[y,x]\in\langle x\rangle_P.
\]
By Proposition~\ref{prop:one-generated}, $\langle x\rangle_P$ is exactly
the associative subalgebra generated by $x$, and hence
$\langle x\rangle_P\subseteq B$. Thus $[x,y]\in B$, and $B$ is a
Poisson subalgebra. Since $P$ is Dedekind, $B$ is a Poisson ideal, in
particular $PB\subseteq B$.
\end{proof}

A not necessarily associative algebra is called \emph{power-associative}
if every one-generated subalgebra is associative.  Following the
classical terminology used in this circle of problems, we shall call
an algebra an \emph{$H$-algebra} if every subalgebra is a two-sided
ideal.

\begin{proposition}[Relation with Outcalt outside characteristic $2$]
\label{prop:outcalt-bridge}
Assume $\operatorname{char}(F)\neq2$ and define on a Poisson algebra
$P$ the bilinear product
\[
x*y=xy+[x,y].
\]
Then $(P,*)$ is power-associative.  A vector subspace of $P$ is a
$*$-subalgebra if and only if it is a Poisson subalgebra, and it is a
two-sided $*$-ideal if and only if it is a Poisson ideal.  Consequently,
$P$ is a Dedekind Poisson algebra if and only if $(P,*)$ is an
$H$-algebra.
\end{proposition}

\begin{proof}
Since $\operatorname{char}(F)\neq2$,
\[
xy=\frac{x*y+y*x}{2},
\qquad
[x,y]=\frac{x*y-y*x}{2}.
\]
Thus closure under $*$ is equivalent to simultaneous closure under the
associative product and the Lie bracket, and the same formulas show
that two-sided $*$-stability is equivalent to the two Poisson ideal
conditions.  Finally, Proposition~\ref{prop:one-generated} shows that
on every one-generated Poisson subalgebra the bracket is zero, so $*$
coincides there with the associative multiplication.  Hence $(P,*)$ is
power-associative.
\end{proof}

\begin{remark}[Agreement with the classical classification]
\label{rem:classical-h-algebras}
Outcalt's theorem \cite{Outcalt1967} therefore applies directly to
Dedekind Poisson algebras when $\operatorname{char}(F)\neq2$.  In its
non-zero nil component the product has the form
\[
u*v=\alpha(u,v)c
\]
with $c$ in the annihilator and $\alpha(v,v)\neq0$ for $v\neq0$.
Writing
\[
\beta(u,v)=\frac{\alpha(u,v)+\alpha(v,u)}2,
\qquad
\omega(u,v)=\frac{\alpha(u,v)-\alpha(v,u)}2
\]
gives precisely the pair of forms occurring in our canonical form
below.  Thus the classification obtained here agrees with Outcalt's
classification outside characteristic $2$.  This agreement also gives
an independent consistency check on the direct Poisson argument.  We
nevertheless keep the direct proof from the two Poisson operations: it
uses a different route through cyclic Poisson subalgebras and, more
importantly, remains valid in characteristic $2$, where the product
$*$ no longer determines the two operations separately.
\end{remark}

The first step is that cyclic associative subalgebras are always
finite-dimensional.

\begin{lemma}
\label{lem:cyclic-finite}
Let $P$ be Dedekind and $a\in P$. Then the associative subalgebra
$\langle a\rangle_{P(+,\cdot)}$ is finite-dimensional.
\end{lemma}

\begin{proof}
By Proposition~\ref{prop:assoc-subalgebra-ideal}, the associative
subalgebra generated by $a^2$ is an ideal. Therefore
\[
a^3\in\langle a^2\rangle_{P(+,\cdot)}
 =\operatorname{span}_F\{a^2,a^4,a^6,\ldots\}.
\]
Thus
\[
a^3=\lambda_1a^2+\lambda_2a^4+\cdots+\lambda_ma^{2m}
\]
for some finite set of scalars. This is a non-zero polynomial relation
among the positive powers of $a$. If all $\lambda_i$ are zero, then
$a^3=0$. Otherwise let $d$ be the largest exponent occurring with a
non-zero coefficient in the displayed relation. If $d>3$, we solve for
$a^d$ in terms of lower powers; if $d=3$, the displayed equality already
expresses $a^3$ in terms of lower powers. Multiplying the resulting
relation by further powers of $a$ shows recursively that all sufficiently
high powers lie in the span of finitely many lower ones. Hence the cyclic
associative subalgebra is finite-dimensional.
\end{proof}

We shall use the following standard finite-dimensional fact.

\begin{lemma}
\label{lem:A2A-identity}
Let $A$ be a finite-dimensional commutative associative algebra. If
$A^2=A$, then $A$ has an identity element.
\end{lemma}

\begin{proof}
Choose a basis $a_1,\ldots,a_n$ of $A$. Since $A=A^2$, for every $i$
there exist $c_{ij}\in A$ such that
\[
a_i=\sum_{j=1}^n c_{ij}a_j.
\]
Let $C=(c_{ij})$ and work in the unitization $F1\oplus A$. Then
\[
(I-C)
\begin{pmatrix}
a_1\\ \vdots\\ a_n
\end{pmatrix}=0.
\]
Multiplication by the adjugate matrix yields
\[
\det(I-C)a_i=0
\qquad(1\leq i\leq n).
\]
The constant term of $\det(I-C)$ is $1$, while every other term belongs
to $A$. Hence $\det(I-C)=1-e$ for some $e\in A$. Thus
$(1-e)a_i=0$ for every $i$, so $ea=a$ for every $a\in A$.
\end{proof}

\begin{lemma}
\label{lem:nonnil-idempotent}
Let $P$ be a Dedekind Poisson algebra. If $P(+,\cdot)$ contains a
non-nilpotent element, then it contains a non-zero idempotent.
\end{lemma}

\begin{proof}
Let $a$ be non-nilpotent and let $B$ be the associative subalgebra
generated by $a$. By Lemma~\ref{lem:cyclic-finite}, $B$ is
finite-dimensional. The descending chain
\[
B\supseteq B^2\supseteq B^3\supseteq\cdots
\]
stabilizes. Since $a$ is not nilpotent, its stable term is non-zero.
Thus for some $m$,
\[
C=B^m=B^{m+1}=\cdots\neq0.
\]
Then $C^2=B^{2m}=C$, so Lemma~\ref{lem:A2A-identity} gives an identity
$e$ of $C$. In particular $e^2=e\neq0$.
\end{proof}

The Dedekind condition allows at most one non-zero idempotent.

\begin{lemma}
\label{lem:unique-idempotent}
A Dedekind Poisson algebra contains at most one non-zero idempotent of
its associated associative algebra.
\end{lemma}

\begin{proof}
Suppose $e$ and $f$ are distinct non-zero idempotents. Then they are
linearly independent: if $f=\lambda e$, the equality $f^2=f$ gives
$\lambda=1$. The one-dimensional associative subalgebras $Fe$ and
$Ff$ are ideals by Proposition~\ref{prop:assoc-subalgebra-ideal}, so
\[
ef\in Fe\cap Ff=0.
\]
Hence $e+f$ is an idempotent. The subalgebra $F(e+f)$ is again an
ideal, and therefore
\[
e=e(e+f)\in F(e+f),
\]
contrary to the linear independence of $e$ and $f$.
\end{proof}

\begin{proposition}[Idempotent--nil decomposition]
\label{prop:idempotent-nil}
Let $P$ be a Dedekind Poisson algebra. Then
\[
P(+,\cdot)=E\oplus N,
\]
where either $E=0$, or
\[
E=Fe,
\qquad e^2=e,
\]
and in the latter case $eN=0$. Every element of $N$ is nilpotent.
\end{proposition}

\begin{proof}
If $P(+,\cdot)$ has no non-zero idempotent, Lemma~\ref{lem:nonnil-idempotent}
shows that every element is nilpotent. We take $E=0$ and $N=P$.

Suppose that $e$ is a non-zero idempotent. Since $Fe$ is an associative
subalgebra, it is an ideal. Hence $eP\subseteq Fe$, and because
$e=e^2\in eP$, we have $eP=Fe$. The linear map
\[
L_e:P\longrightarrow P,
\qquad
L_e(x)=ex,
\]
is idempotent. Therefore
\[
P=\operatorname{Im}L_e\oplus\ker L_e
 =Fe\oplus N,
\qquad
N=\ker L_e,
\]
and $eN=0$. The subspace $N$ is an associative ideal.

If some $n\in N$ were non-nilpotent, Lemma~\ref{lem:nonnil-idempotent},
applied to the cyclic subalgebra generated by $n$, would produce a
non-zero idempotent $f\in N$. This contradicts
Lemma~\ref{lem:unique-idempotent}. Thus every element of $N$ is
nilpotent.
\end{proof}

\section{The nilpotent part}
\label{sec:nilpart}

We now analyze the nil algebra $N$ from
Proposition~\ref{prop:idempotent-nil}. The arguments in this section do
not use finite-dimensionality and do not divide by $2$.

\begin{lemma}
\label{lem:x3zero}
For every $x\in N$,
\[
x^3=0.
\]
\end{lemma}

\begin{proof}
The assertion is immediate for $x=0$, so assume $x\neq0$. Let $m$ be
the nilpotency index of $x$, so $x^m=0$ and $x^{m-1}\neq0$. If
$m\geq4$, then
\[
x,x^2,\ldots,x^{m-1}
\]
are linearly independent. Indeed, if
$\sum_{i=1}^{m-1}\alpha_i x^i=0$ and $k$ is the least index with
$\alpha_k\neq0$, multiplication by $x^{m-1-k}$ gives
$\alpha_kx^{m-1}=0$, a contradiction.

On the other hand, the associative subalgebra generated by $x^2$ is an
ideal, so
\[
x^3\in\langle x^2\rangle_{P(+,\cdot)}
 =\operatorname{span}_F\{x^2,x^4,x^6,\ldots\}.
\]
After deleting the powers which are already zero, this contradicts the
linear independence above. Hence $m\leq3$.
\end{proof}

\begin{lemma}
\label{lem:ann-squarezero}
For $x\in N$,
\[
x^2=0
\quad\Longleftrightarrow\quad
x\in\operatorname{Ann}(N).
\]
Consequently,
\[
\operatorname{Ann}(N)=\{x\in N\mid x^2=0\}.
\]
\end{lemma}

\begin{proof}
Only the forward implication is non-trivial. Suppose $x^2=0$. Then
$Fx$ is an associative subalgebra and hence an associative ideal. Thus
for every $y\in N$,
\[
xy=\lambda x
\]
for some $\lambda\in F$. Since $y$ is nilpotent, $y^m=0$ for some $m$.
Associativity gives
\[
0=xy^m=\lambda^m x.
\]
If $x\neq0$, then $\lambda=0$; the assertion is trivial when $x=0$.
Hence $xN=0$.
\end{proof}

\begin{lemma}
\label{lem:xN-line}
If $x\in N$ and $x^2\neq0$, then
\[
xN\subseteq Fx^2.
\]
\end{lemma}

\begin{proof}
By Lemma~\ref{lem:x3zero},
\[
\langle x\rangle_{P(+,\cdot)}=Fx+Fx^2.
\]
This subalgebra is an ideal, so for $y\in N$ we may write
\[
xy=\alpha x+\beta x^2.
\]
Now $(x^2)^2=x^4=0$, and Lemma~\ref{lem:ann-squarezero} gives
$x^2\in\operatorname{Ann}(N)$. Hence
\[
0=x^2y=x(xy)=\alpha x^2+\beta x^3=\alpha x^2.
\]
Since $x^2\neq0$, we have $\alpha=0$.
\end{proof}

The square of $N$ is therefore either zero or one-dimensional.

\begin{proposition}
\label{prop:N2}
Either $N^2=0$, or
\[
\dim_F N^2=1
\qquad\text{and}\qquad
N^3=0.
\]
More precisely, if $a^2=c\neq0$, then
\[
N^2=Fc
\qquad\text{and}\qquad
c\in\operatorname{Ann}(N).
\]
\end{proposition}

\begin{proof}
Assume $N^2\neq0$. Since an associative algebra with $x^2=0$ for all
$x$ may still have non-zero products in characteristic $2$, we first
show that some square is non-zero. If $x^2=0$ for every $x\in N$, then
Lemma~\ref{lem:ann-squarezero} gives $N=\operatorname{Ann}(N)$, and
hence $N^2=0$, a contradiction. Thus choose $a\in N$ with
\[
a^2=c\neq0.
\]
By Lemma~\ref{lem:xN-line}, $aN\subseteq Fc$.

Let $y\in N$. If $y^2=0$, then
$y\in\operatorname{Ann}(N)$ and $ay=0$. Suppose $y^2\neq0$.
If $ay\neq0$, then
\[
ay\in Fa^2\cap Fy^2,
\]
where the second inclusion follows by applying
Lemma~\ref{lem:xN-line} to $y$. Hence $Fy^2=Fc$.

It remains to consider $y^2\neq0$ and $ay=0$. Put $x=a+y$. If
$x^2=0$, Lemma~\ref{lem:ann-squarezero} would give $xN=0$, whereas
\[
xa=a^2+ay=c\neq0.
\]
Thus $x^2\neq0$. Lemma~\ref{lem:xN-line} applied to $x$ gives
\[
c=xa\in Fx^2.
\]
Hence $Fx^2=Fc$. Since $ay=0$,
\[
x^2=(a+y)^2=c+y^2,
\]
so $y^2\in Fc$.

We have proved that every square in $N$ lies in $Fc$. For arbitrary
$u,v\in N$, if $u^2=0$ then $uv=0$ by
Lemma~\ref{lem:ann-squarezero}; if $u^2\neq0$, then
$uv\in Fu^2\subseteq Fc$ by Lemma~\ref{lem:xN-line}. Therefore
$N^2\subseteq Fc$. Since $c=a^2\in N^2$, we have $N^2=Fc$.
Finally, $c^2=a^4=0$, so
$c\in\operatorname{Ann}(N)$ and $N^3=N^2N=0$.
\end{proof}

We collect the associative part of the argument.

\begin{theorem}[Associative structure theorem]
\label{thm:assoc-structure}
Let $P$ be a Dedekind Poisson algebra over an arbitrary field. Then
\[
P(+,\cdot)=E\oplus N,
\]
where
\[
E=0
\quad\text{or}\quad
E=Fe,\qquad e^2=e,
\]
and, when $E\neq0$,
\[
eN=0.
\]
Moreover,
\[
N^3=0,
\qquad
\dim_F N^2\leq1,
\]
and
\[
\operatorname{Ann}(N)=\{x\in N\mid x^2=0\}.
\]
The non-zero idempotent $e$, when it exists, is unique.
\end{theorem}

\begin{proof}
The decomposition and uniqueness of $e$ follow from
Proposition~\ref{prop:idempotent-nil} and
Lemma~\ref{lem:unique-idempotent}. Lemma~\ref{lem:ann-squarezero} gives
the description of the annihilator, and
Proposition~\ref{prop:N2} gives $N^3=0$ and
$\dim_FN^2\leq1$.
\end{proof}

\section{The Lie multiplication and the canonical form}
\label{sec:lie}

We next determine how much Lie multiplication can survive the
associative restrictions obtained above.

We first recall a useful fact valid for every Poisson algebra.

\begin{lemma}
\label{lem:idempotent-central}
Every idempotent of $P(+,\cdot)$ belongs to
$\zeta(P(+,[\, ,\,]))$.
\end{lemma}

\begin{proof}
Let $e^2=e$ and $x\in P$. Put $d=[e,x]$. The Leibniz identity gives
\[
d=[e^2,x]=2e[e,x]=2ed.
\]
Multiplying by $e$ yields $ed=2ed$, and therefore $ed=0$. Hence
$d=2ed=0$. In characteristic $2$, the equality $d=2ed$ already gives
$d=0$.
\end{proof}

\begin{proposition}
\label{prop:derived-in-N2}
Let
\[
P(+,\cdot)=E\oplus N
\]
be the decomposition of Theorem~\ref{thm:assoc-structure}. Then $N$ is
a Poisson ideal and
\[
[P,P]=[N,N]\subseteq N^2.
\]
\end{proposition}

\begin{proof}
If $E=0$, then $N=P$. Suppose $E=Fe$. By
Lemma~\ref{lem:idempotent-central}, $[e,P]=0$. Since $N=\ker L_e$,
for $n\in N$ and $p\in P$ the Leibniz identity gives
\[
0=[en,p]=e[n,p]+n[e,p]=e[n,p].
\]
Thus $[n,p]\in N$, and $N$ is a Poisson ideal. Moreover
$[P,P]=[N,N]$ because $e$ is Lie-central.

By Proposition~\ref{prop:derived-in-square},
$[P,P]\subseteq P^2$. If $E=Fe$, then
\[
P^2=Fe\oplus N^2,
\]
while $[N,N]\subseteq N$. Hence
\[
[N,N]\subseteq N\cap P^2=N^2.
\]
The case $E=0$ is immediate because $P=N$.
\end{proof}

\begin{corollary}
\label{cor:N2zero-Liezero}
If $N^2=0$, then $[P,P]=0$.
\end{corollary}

Assume now that $N^2=Fc\neq0$. The entire associative annihilator of
$N$ becomes Lie-central.

\begin{lemma}
\label{lem:c-central}
If $N^2=Fc\neq0$, then
\[
Fc\subseteq\zeta(P(+,[\, ,\,])).
\]
\end{lemma}

\begin{proof}
By Proposition~\ref{prop:N2}, there exists $a\in N$ with $a^2\neq0$.
Write
\[
a^2=\mu c,
\qquad
\mu\in F^\times.
\]
For every $x\in P$, Proposition~\ref{prop:derived-in-N2} gives
$[a,x]\in Fc$. Since $cN=0$, we have $a[a,x]=0$, and therefore
\[
\mu[c,x]=[a^2,x]=2a[a,x]=0.
\]
As $\mu\neq0$, it follows that $[c,x]=0$ for every $x\in P$.
No division by $2$ is involved, so the argument also applies in
characteristic $2$.
\end{proof}

\begin{lemma}[Central annihilator lemma]
\label{lem:ann-central}
If $N^2=Fc\neq0$, then
\[
\operatorname{Ann}(N)
 \subseteq\zeta(P(+,[\, ,\,])).
\]
\end{lemma}

\begin{proof}
Let $z\in\operatorname{Ann}(N)$. Then $z^2=0$ by
Lemma~\ref{lem:ann-squarezero}, and hence
\[
\langle z\rangle_P=Fz
\]
by Proposition~\ref{prop:one-generated}. The cyclic criterion implies
that $Fz$ is a Poisson ideal, so
\[
[P,z]\subseteq Fz.
\]
On the other hand, Proposition~\ref{prop:derived-in-N2} gives
\[
[P,z]\subseteq[P,P]\subseteq Fc.
\]
If $z\notin Fc$, then $Fz\cap Fc=0$ and $[P,z]=0$. If $z\in Fc$, the
same conclusion follows from Lemma~\ref{lem:c-central}.
\end{proof}

We now introduce the bilinear forms which encode the remaining
structure. Recall that a bilinear form $\beta:V\times V\to F$ is
\emph{symmetric} if $\beta(u,v)=\beta(v,u)$, and an alternating
bilinear form $\omega$ satisfies $\omega(v,v)=0$ for every $v\in V$.
For a symmetric form we shall say that the associated quadratic map
$v\mapsto\beta(v,v)$ is \emph{anisotropic} when
\[
\beta(v,v)\neq0
\qquad\text{for every }0\neq v\in V.
\]
We shall always display this explicit condition, so no theory of
quadratic forms is required.

\begin{proposition}[Canonical form]
\label{prop:canonical-form}
Let $P$ be Dedekind and suppose $N^2=Fc\neq0$. There exist vector
subspaces $V,Z\leq N$ such that
\[
N=V\oplus Z\oplus Fc,
\qquad
\operatorname{Ann}(N)=Z\oplus Fc.
\]
There are bilinear forms
\[
\beta:V\times V\longrightarrow F,
\qquad
\omega:V\times V\longrightarrow F
\]
such that
\[
uv=\beta(u,v)c,
\qquad
[u,v]=\omega(u,v)c
\qquad(u,v\in V).
\]
The form $\beta$ is symmetric and satisfies
\[
\beta(v,v)\neq0
\qquad(0\neq v\in V),
\]
while $\omega$ is alternating. Moreover,
\[
ZP=Fc\,P=0
\]
and
\[
[Z,P]=[c,P]=0.
\]
\end{proposition}

\begin{proof}
Choose a vector-space complement $Z$ of $Fc$ in
$\operatorname{Ann}(N)$, and then choose a complement $V$ of
$\operatorname{Ann}(N)$ in $N$. Thus
\[
N=V\oplus Z\oplus Fc.
\]
Since $N^2=Fc$, for $u,v\in V$ there is a unique scalar
$\beta(u,v)$ such that
\[
uv=\beta(u,v)c.
\]
Bilinearity and commutativity of the associative multiplication show
that $\beta$ is a symmetric bilinear form. If $0\neq v\in V$ and
$\beta(v,v)=0$, then $v^2=0$, so
$v\in\operatorname{Ann}(N)$ by Lemma~\ref{lem:ann-squarezero}, contrary
to the choice of $V$. Thus $\beta(v,v)\neq0$ for non-zero $v$.

By Proposition~\ref{prop:derived-in-N2}, $[P,P]\subseteq Fc$. Hence
there is a bilinear form $\omega$ on $V$ such that
\[
[u,v]=\omega(u,v)c.
\]
It is alternating because the Lie multiplication is alternating.
Finally, $Z\oplus Fc=\operatorname{Ann}(N)$, and
Lemma~\ref{lem:ann-central} gives the asserted vanishing of products
and brackets involving $Z$ or $c$. The products with $e$, when $e$
exists, are zero by Theorem~\ref{thm:assoc-structure}, and the brackets
with $e$ are zero by Lemma~\ref{lem:idempotent-central}.
\end{proof}

\begin{remark}
\label{rem:no-compatibility}
No additional compatibility relation between $\beta$ and $\omega$ is
hidden in Proposition~\ref{prop:canonical-form}. Indeed, if both forms
take values in the line $Fc$, with $cP=0$ and $[c,P]=0$, then for
$u,v,w\in V$,
\[
[uv,w]=0
\]
and
\[
u[v,w]+v[u,w]=0.
\]
Thus the Leibniz identity is automatic. The Jacobi identity is likewise
automatic because $[V,V]\subseteq Fc$ and $[c,P]=0$.
\end{remark}

\section{Classification of Dedekind Poisson algebras}
\label{sec:classification}

We can now state the main theorem. In the multiplication tables below,
all products and brackets not explicitly displayed or forced by
bilinearity, commutativity of the associative multiplication, and
alternation of the Lie multiplication are understood to be zero.

\begin{theorem}[Main classification theorem]
\label{thm:main}
Let $P$ be a Poisson algebra over an arbitrary field $F$. Then $P$ is a
Dedekind Poisson algebra if and only if it is of one of the following
two types.

\medskip
\noindent\textup{\rm(I)}
\[
P=E\oplus Z,
\]
where the decomposition is a direct sum of vector spaces and $Z$ is a
zero Poisson algebra,
\[
E=0
\quad\text{or}\quad
E=Fe,
\qquad e^2=e,
\]
and all Lie brackets are zero.

\medskip
\noindent\textup{\rm(II)}
\[
P=E\oplus Z\oplus V\oplus Fc,
\qquad V\neq0,
\]
where the decomposition is a direct sum of vector spaces, $Fc$ is
one-dimensional with $c\neq0$, and
\[
E=0
\quad\text{or}\quad
E=Fe,
\qquad e^2=e,
\]
and there exist a symmetric bilinear form
\[
\beta:V\times V\to F
\]
and an alternating bilinear form
\[
\omega:V\times V\to F
\]
such that
\[
\beta(v,v)\neq0
\qquad\text{for every }0\neq v\in V,
\]
and the only possibly non-zero operations besides $e^2=e$ are
\[
uv=\beta(u,v)c,
\qquad
[u,v]=\omega(u,v)c
\qquad(u,v\in V).
\]
In particular,
\[
ZP=Fc\,P=0,
\qquad
[Z,P]=[c,P]=[e,P]=0.
\]
Conversely, every algebra of type \textup{\rm(I)} or \textup{\rm(II)}
is a Dedekind Poisson algebra.
\end{theorem}

\begin{proof}
Let first $P$ be Dedekind. By Theorem~\ref{thm:assoc-structure},
\[
P(+,\cdot)=E\oplus N,
\]
where $E=0$ or $E=Fe$, $e^2=e$, $eN=0$, and
$\dim_FN^2\leq1$.

If $N^2=0$, Corollary~\ref{cor:N2zero-Liezero} gives $[P,P]=0$.
Taking $Z=N$ yields type~\textup{\rm(I)}.

Suppose $N^2\neq0$. Write $N^2=Fc$. Proposition~\ref{prop:canonical-form}
provides
\[
N=V\oplus Z\oplus Fc
\]
and the required forms $\beta$ and $\omega$. Thus $P$ is of
type~\textup{\rm(II)}.

Conversely, let $P$ be an algebra of type~\textup{\rm(I)} or
\textup{\rm(II)}. The displayed multiplication is commutative and
associative. In type~\textup{\rm(II)}, the Jacobi and Leibniz identities
follow from Remark~\ref{rem:no-compatibility}; type~\textup{\rm(I)} is
immediate. Hence $P$ is a Poisson algebra.

By Lemma~\ref{lem:cyclic-criterion}, it suffices to prove that every
cyclic Poisson subalgebra is an ideal. In type~\textup{\rm(I)}, let
$x=\alpha e+z$ (with the $e$-term omitted when $E=0$). If $\alpha=0$,
then $\langle x\rangle_P=Fx\subseteq Z$ is an ideal. If
$\alpha\neq0$, then $x^2=\alpha^2e$, so
$e,z\in\langle x\rangle_P$ and
$\langle x\rangle_P=Fe+Fz$, again an ideal. Thus type~\textup{\rm(I)}
is Dedekind.

Now consider type~\textup{\rm(II)} and write
\[
x=\alpha e+z+v+\gamma c,
\]
where the term $\alpha e$ is omitted when $E=0$, $z\in Z$, $v\in V$,
and $\gamma\in F$.

Suppose first that $\alpha=0$. If $v=0$, then $x$ lies in the central
associative annihilator $Z\oplus Fc$, so
\[
\langle x\rangle_P=Fx
\]
is a Poisson ideal. If $v\neq0$, then
\[
x^2=\beta(v,v)c\neq0,
\qquad
x^3=0.
\]
Therefore
\[
\langle x\rangle_P=Fx+Fc.
\]
For every $p\in P$, both $px$ and $[p,x]$ lie in $Fc$, while
$Pc=[P,c]=0$. Hence $\langle x\rangle_P$ is a Poisson ideal.

Now let $\alpha\neq0$. Put
\[
n=z+v+\gamma c,
\qquad
x=\alpha e+n.
\]
Since $en=0$ and $n^3=0$,
\[
x^3=\alpha^3e.
\]
Thus $e\in\langle x\rangle_P$ and consequently
$n=x-\alpha e\in\langle x\rangle_P$. If $v=0$, then $n$ is in the
central annihilator and
\[
\langle x\rangle_P=Fe+Fn,
\]
a sum of Poisson ideals. If $v\neq0$, then
\[
n^2=\beta(v,v)c\neq0,
\]
so $c\in\langle x\rangle_P$ and
\[
\langle x\rangle_P=Fe+Fn+Fc.
\]
Again $Pn,[P,n]\subseteq Fc$, while $Fe$ and $Fc$ are Poisson ideals.
Hence $\langle x\rangle_P$ is a Poisson ideal. The cyclic criterion now
shows that $P$ is Dedekind.
\end{proof}

\begin{remark}
The two types in Theorem~\ref{thm:main} are disjoint up to isomorphism.
Indeed, the intrinsic ideal $N$ consisting of the associative
nilpotent elements satisfies $N^2=0$ in type~\textup{\rm(I)} and
$N^2=Fc\neq0$ in type~\textup{\rm(II)}.
\end{remark}

\begin{corollary}
\label{cor:derived-small}
For every Dedekind Poisson algebra $P$,
\[
\dim_F[P,P]\leq1
\]
and
\[
[P,P]\subseteq
P^2\cap\operatorname{Ann}(P(+,\cdot))
\cap\zeta(P(+,[\, ,\,])).
\]
In particular,
\[
[[P,P],P]=0,
\]
so the associated Lie algebra is nilpotent of class at most $2$.
\end{corollary}

\begin{proof}
In type~\textup{\rm(I)}, $[P,P]=0$. In type~\textup{\rm(II)},
$[P,P]\subseteq Fc$, and $Fc$ lies simultaneously in $P^2$, the
associative annihilator, and the Lie center.
\end{proof}

\begin{corollary}[Unital case]
\label{cor:unital}
Every non-zero unital Dedekind Poisson algebra is isomorphic to $F$,
with its ordinary associative multiplication and zero Lie
multiplication.
\end{corollary}

\begin{proof}
The one-generated Poisson subalgebra $F1=\langle1\rangle_P$ is a
Poisson ideal. Hence
\[
P=P1\subseteq F1,
\]
so $P=F1$. The Lie multiplication on a one-dimensional space is zero
by alternation.
\end{proof}

\section{Isomorphisms}
\label{sec:isomorphism}

We now determine when two algebras in Theorem~\ref{thm:main} are
isomorphic. All isomorphisms are understood to be $F$-linear.
Thus the isomorphism problem reduces to the simultaneous similarity
of the two bilinear forms, together with the dimension of the central
zero component and the presence or absence of the idempotent component.

The ideal $N$ is intrinsically determined: it is precisely the set of
associative nilpotent elements of $P$. In type~\textup{\rm(II)}, the
bilinear forms are intrinsically defined on
$N/\operatorname{Ann}(N)$ with values in the one-dimensional space
$N^2$. Choosing a non-zero generator $c$ of $N^2$ identifies these
$N^2$-valued forms with the $F$-valued forms $\beta$ and $\omega$.
Changing that generator is exactly the source of the common scalar in
the simultaneous-similarity relation below.

Two pairs of bilinear forms $(\beta,\omega)$ on $V$ and
$(\beta',\omega')$ on $V'$ will be called \emph{simultaneously similar}
if there exist a vector-space isomorphism
\[
\varphi:V\longrightarrow V'
\]
and a scalar $\lambda\in F^\times$ such that
\[
\beta'(\varphi u,\varphi v)=\lambda\beta(u,v)
\]
and
\[
\omega'(\varphi u,\varphi v)=\lambda\omega(u,v)
\]
for all $u,v\in V$.

\begin{theorem}[Isomorphism theorem]
\label{thm:isomorphism}
Let $P$ and $P'$ be Dedekind Poisson algebras over $F$ in the canonical
forms of Theorem~\ref{thm:main}.

For type~\textup{\rm(I)}, the algebras are isomorphic if and only if
they simultaneously have, or simultaneously do not have, the
idempotent summand $Fe$, and
\[
\dim_F Z=\dim_F Z'.
\]

For type~\textup{\rm(II)}, write
\[
P=P(E,Z,V,c,\beta,\omega),
\qquad
P'=P(E',Z',V',c',\beta',\omega').
\]
Then $P\cong P'$ if and only if
\[
E=0\Longleftrightarrow E'=0,
\qquad
\dim_F Z=\dim_F Z',
\]
and the pairs $(\beta,\omega)$ and $(\beta',\omega')$ are
simultaneously similar.
\end{theorem}

\begin{proof}
Consider first type~\textup{\rm(II)}. Let
$f:P\to P'$ be a Poisson isomorphism. The existence of a non-zero
idempotent is invariant under isomorphism, and when it exists it is
unique by Lemma~\ref{lem:unique-idempotent}; hence $f(e)=e'$.
Consequently the nilpotent summand $N$ is carried onto $N'$.

The subspaces
\[
N^2=Fc
\qquad\text{and}\qquad
\operatorname{Ann}(N)=Z\oplus Fc
\]
are intrinsically determined by the associative multiplication, so
$f$ carries them onto $N'^2=Fc'$ and
$\operatorname{Ann}(N')=Z'\oplus Fc'$. Therefore
\[
\dim_FZ
 =\dim_F\bigl(\operatorname{Ann}(N)/N^2\bigr)
 =\dim_F\bigl(\operatorname{Ann}(N')/N'^2\bigr)
 =\dim_FZ'.
\]
There is a scalar $\lambda\in F^\times$ such that
\[
f(c)=\lambda c'.
\]
The map $f$ also induces a vector-space isomorphism
\[
\varphi:N/\operatorname{Ann}(N)
 \longrightarrow
N'/\operatorname{Ann}(N'),
\]
which, after using the chosen complements, is identified with an
isomorphism $\varphi:V\to V'$.

For $u,v\in V$,
\[
f(uv)=f(\beta(u,v)c)=\lambda\beta(u,v)c'.
\]
Central-annihilator components added to representatives do not affect
products, so also
\[
f(u)f(v)=\beta'(\varphi u,\varphi v)c'.
\]
Hence
\[
\beta'(\varphi u,\varphi v)=\lambda\beta(u,v).
\]
The same argument with the Lie bracket gives
\[
\omega'(\varphi u,\varphi v)=\lambda\omega(u,v).
\]
Thus the conditions are necessary.

Conversely, suppose they hold. Choose any vector-space isomorphism
$\theta:Z\to Z'$. Define
\[
f(e)=e',
\qquad
f(c)=\lambda c',
\qquad
f(z)=\theta(z),
\qquad
f(v)=\varphi(v),
\]
omitting $e$ when the idempotent summand is absent. The simultaneous
similarity identities show directly that $f$ preserves both
multiplications. Hence it is a Poisson isomorphism.

The type~\textup{\rm(I)} statement is the same argument without the
forms. The existence of the unique non-zero idempotent distinguishes
the two possibilities, and the remaining summand is a zero algebra,
classified by its vector-space dimension.
\end{proof}

\begin{remark}
Theorem~\ref{thm:isomorphism} is a structural classification up to
isomorphism over an arbitrary field. It deliberately does not attempt
to replace simultaneous similarity of $(\beta,\omega)$ by field-specific
matrix normal forms. Such normal forms would depend strongly on $F$ and
would obscure the arbitrary-field nature of the result.
\end{remark}

\section{Ground fields, mixed algebras and examples}
\label{sec:fields}

We now make the dependence on the ground field explicit.

\begin{definition}
A Poisson algebra $P$ is called \emph{mixed} if both of its
multiplications are non-zero, equivalently
\[
P^2\neq0
\qquad\text{and}\qquad
[P,P]\neq0.
\]
For $a\in F^\times$, we call $a$ a \emph{square} if $a=b^2$ for some
$b\in F^\times$, and a \emph{non-square} otherwise. In characteristic
$2$, write
\[
F^2=\{a^2:a\in F\},
\]
which is a subfield of $F$. The field $F$ is called \emph{perfect} if
the Frobenius map $x\mapsto x^2$ is surjective, equivalently if
$F^2=F$; otherwise $F$ is called \emph{imperfect}.
\end{definition}

A mixed algebra in Theorem~\ref{thm:main} must have a non-zero
alternating form $\omega$, hence $\dim_FV\geq2$.

\begin{proposition}[A characteristic-$2$ dimension bound]
\label{prop:char2-dim}
Assume $\operatorname{char}(F)=2$ and let $V$ be the active component
of a type~\textup{\rm(II)} Dedekind Poisson algebra. Then
\[
\dim_FV\leq [F:F^2].
\]
More precisely, for any basis $(v_i)_{i\in I}$ of $V$, the family
$(\beta(v_i,v_i))_{i\in I}$ is linearly independent over $F^2$.
\end{proposition}

\begin{proof}
Put $a_i=\beta(v_i,v_i)$. For a vector
$v=\sum_i x_iv_i$ with finite support, symmetry and
$\operatorname{char}(F)=2$ give
\[
\beta(v,v)=\sum_i a_ix_i^2.
\]
If the family $(a_i)$ were linearly dependent over $F^2$, we could
write
\[
\sum_i a_i y_i^2=0
\]
with finitely many $y_i\in F$, not all zero. Then the non-zero vector
$\sum_i y_iv_i$ would satisfy $\beta(v,v)=0$, contrary to the
condition in Theorem~\ref{thm:main}. Conversely, such a relation is
exactly an isotropic vector. Hence the family is $F^2$-linearly
independent, and its cardinality is at most
$\dim_{F^2}F=[F:F^2]$.
\end{proof}

The existence of the symmetric form $\beta$ on a space of dimension
at least $2$ now has a simple field-theoretic criterion.

\begin{theorem}[Field criterion for mixed algebras]
\label{thm:field-criterion}
Let $F$ be a field.

\begin{enumerate}
\item If $\operatorname{char}(F)\neq2$, then a mixed Dedekind Poisson
algebra over $F$ exists if and only if $F$ has a non-square.

\item If $\operatorname{char}(F)=2$, then a mixed Dedekind Poisson
algebra over $F$ exists if and only if $F$ is imperfect.
\end{enumerate}
\end{theorem}

\begin{proof}
Suppose first that $\operatorname{char}(F)\neq2$ and a mixed Dedekind
Poisson algebra exists. By Theorem~\ref{thm:main}, there is a subspace
$V$ of dimension at least $2$ carrying a symmetric bilinear form
$\beta$ such that $\beta(v,v)\neq0$ for every $0\neq v\in V$.
Choose linearly independent $u,v\in V$ and write
\[
a=\beta(u,u),
\qquad
b=\beta(u,v),
\qquad
c=\beta(v,v).
\]
Then $a\neq0$ and
\[
\beta(tu+v,tu+v)=at^2+2bt+c
\]
never vanishes. If every non-zero element of $F$ were a square, the
discriminant
\[
4(b^2-ac)
\]
would either be zero or a square, and the quadratic polynomial would
have a root in $F$, a contradiction. Hence $F$ has a non-square.

Conversely, let $d\in F^\times$ be a non-square. On
$V=Fu\oplus Fv$, define a symmetric form by
\[
\beta(u,u)=1,
\qquad
\beta(v,v)=-d,
\qquad
\beta(u,v)=0.
\]
Then
\[
\beta(xu+yv,xu+yv)=x^2-dy^2,
\]
which is zero only for $x=y=0$. Let $P=V\oplus Fc$ and define
\[
uv=\beta(u,v)c
\]
using this form, and choose the alternating form $\omega$ with
\[
\omega(u,v)=1.
\]
Theorem~\ref{thm:main} gives a mixed Dedekind Poisson algebra.

Now let $\operatorname{char}(F)=2$. Suppose first that $F$ is perfect
and $V$ has dimension at least $2$. For independent $u,v\in V$, put
$a=\beta(u,u)$ and $c=\beta(v,v)$. Both are non-zero, and symmetry gives
\[
\beta(xu+yv,xu+yv)=ax^2+cy^2
\]
because the cross term is $2xy\beta(u,v)=0$. Since $F$ is perfect,
$c/a=s^2$ for some $s\in F$. Then
\[
\beta(su+v,su+v)=as^2+c=c+c=0,
\]
a contradiction. Thus $\dim_FV\leq1$, and every alternating form on
$V$ is zero. Hence no mixed Dedekind Poisson algebra exists.

Conversely, if $F$ is imperfect, choose $d\in F\setminus F^2$ and let
$V=Fu\oplus Fv$. Define
\[
\beta(u,u)=1,
\qquad
\beta(v,v)=d,
\qquad
\beta(u,v)=0.
\]
Then
\[
\beta(xu+yv,xu+yv)=x^2+dy^2.
\]
If this were zero with $y\neq0$, we would have
$d=(x/y)^2$, contrary to the choice of $d$; if $y=0$, then $x=0$.
Again choose a non-zero alternating form $\omega$ by
$\omega(u,v)=1$. Theorem~\ref{thm:main} yields a mixed Dedekind
Poisson algebra.
\end{proof}

\begin{corollary}
\label{cor:algclosed}
If $F$ is algebraically closed, then every Dedekind Poisson algebra over
$F$ has zero Lie multiplication. Moreover, in the canonical form of
Theorem~\ref{thm:main},
\[
\dim_FV\leq1.
\]
\end{corollary}

\begin{proof}
If $\operatorname{char}(F)\neq2$, every non-zero element of an
algebraically closed field is a square. If $\operatorname{char}(F)=2$,
an algebraically closed field is perfect. The proof of
Theorem~\ref{thm:field-criterion} shows in either case that an
anisotropic $\beta$ cannot exist on a space of dimension at least $2$.
Thus $\dim_FV\leq1$, and every alternating form on $V$ is zero.
\end{proof}

\begin{corollary}
\label{cor:perfect-char2}
Let $\operatorname{char}(F)=2$ and suppose that $F$ is perfect. Then
for every Dedekind Poisson algebra $P$ over $F$,
\[
[P,P]=0.
\]
In particular this holds over every finite field of characteristic
$2$.
\end{corollary}

\begin{proof}
Since $F=F^2$, Proposition~\ref{prop:char2-dim} gives
$\dim_FV\leq1$ in type~\textup{\rm(II)}. Every alternating form on
such a space is zero, and type~\textup{\rm(I)} already has zero Lie
multiplication.
\end{proof}

\begin{corollary}
\label{cor:finite-fields}
Let $F=\mathbb F_q$ be a finite field. If $q$ is even, every Dedekind
Poisson algebra over $F$ has zero Lie multiplication. If $q$ is odd,
mixed Dedekind Poisson algebras exist.
\end{corollary}

\begin{proof}
Finite fields are perfect. If $q$ is odd, the cyclic group
$\mathbb F_q^\times$ has non-squares, so
Theorem~\ref{thm:field-criterion} applies.
\end{proof}

We conclude with explicit mixed examples. They also illustrate why the
arbitrary-field formulation is useful.

\begin{example}[The real and rational fields]
\label{ex:RQ}
Let $F=\mathbb R$ or $F=\mathbb Q$ and let
\[
P=Fa\oplus Fb\oplus Fc.
\]
Define
\[
a^2=c,
\qquad
b^2=c,
\qquad
ab=0,
\]
and
\[
[a,b]=c.
\]
All other basic products and brackets are zero. For
$x=\alpha a+\beta b$,
\[
x^2=(\alpha^2+\beta^2)c.
\]
Over both $\mathbb R$ and $\mathbb Q$, this is zero only when
$\alpha=\beta=0$. Hence Theorem~\ref{thm:main} shows that $P$ is a mixed
Dedekind Poisson algebra with
\[
P^2=[P,P]=Fc.
\]
\end{example}

\begin{remark}[Infinite-dimensional mixed examples]
The real and rational constructions extend to arbitrary cardinal
dimension.  If $I$ is any index set with at least two elements, take
\[
V=\bigoplus_{i\in I}Fv_i,
\qquad F=\mathbb R\ \text{or}\ \mathbb Q,
\]
and define $\beta(v_i,v_j)=\delta_{ij}$.  Every non-zero vector has
finite support, so $\beta(v,v)$ is a non-zero sum of squares.  Choosing
an alternating form $\omega$ which is non-zero on one pair of basis
vectors gives a mixed Dedekind Poisson algebra $V\oplus Fc$ of the
corresponding dimension.  Thus the arbitrary-dimensional statement is
realized by explicit examples and is not merely formal.
\end{remark}

\begin{example}[Finite fields of odd order]
\label{ex:finite-odd}
Let $F=\mathbb F_q$ with $q$ odd and choose a non-square
$d\in F^\times$. Let
\[
P=Fa\oplus Fb\oplus Fc
\]
with
\[
a^2=c,
\qquad
b^2=-dc,
\qquad
ab=0,
\qquad
[a,b]=c.
\]
Then
\[
(\alpha a+\beta b)^2=(\alpha^2-d\beta^2)c.
\]
The coefficient can vanish for a non-zero pair $(\alpha,\beta)$ only
if $d$ is a square. Thus $P$ is a mixed Dedekind Poisson algebra.
For example, over $\mathbb F_3$ one may take $d=2$, which gives
$a^2=b^2=c$.
\end{example}

\begin{example}[An imperfect field of characteristic $2$]
\label{ex:char2}
Let
\[
F=\mathbb F_2(t)
\]
and
\[
P=Fu\oplus Fv\oplus Fc.
\]
Define
\[
u^2=c,
\qquad
v^2=tc,
\qquad
uv=0,
\]
and
\[
[u,v]=c.
\]
The bracket is alternating; in characteristic $2$ this also gives
$[v,u]=c$. For $x=\alpha u+\beta v$,
\[
x^2=(\alpha^2+t\beta^2)c.
\]
If this coefficient vanished for $\beta\neq0$, then
$t=(\alpha/\beta)^2$, impossible in $\mathbb F_2(t)$. Hence the
condition of Theorem~\ref{thm:main} is satisfied and $P$ is mixed.
This shows that characteristic $2$ does not create an additional
structural type; rather, it changes which symmetric forms, and hence
which mixed algebras, can exist over the ground field.
\end{example}

\begin{example}[Characteristic $2$: the off-diagonal part matters]
\label{ex:char2-offdiag}
Let $F=\mathbb F_2(t)$ and, for $s\in F$, let
\[
P_s=Fu\oplus Fv\oplus Fc
\]
with
\[
u^2=c,
\qquad
v^2=tc,
\qquad
uv=sc,
\qquad
[u,v]=c,
\]
and all products and brackets involving $c$ equal to zero.  The
symmetric form on $Fu\oplus Fv$ has matrix
\[
B_s=\begin{pmatrix}1&s\\ s&t\end{pmatrix},
\]
while the alternating form has matrix
\[
W=\begin{pmatrix}0&1\\ 1&0\end{pmatrix}.
\]
For $x=\alpha u+\gamma v$ one has
\[
x^2=(\alpha^2+t\gamma^2)c,
\]
so the condition $x^2\neq0$ for $x\neq0$ is independent of $s$ and
each $P_s$ is Dedekind.  Nevertheless the off-diagonal parameter is
part of the Poisson isomorphism data.  Indeed, under simultaneous
similarity of the two forms the ratio
\[
\frac{\det B_s}{\det W}=t+s^2
\]
is invariant.  Since the Frobenius map is injective on every field,
distinct values of $s$ give distinct values of this invariant.  Here
$[F:F^2]=2$, so Proposition~\ref{prop:char2-dim} is attained sharply.
This illustrates a specifically characteristic-$2$
feature: the diagonal map $v\mapsto\beta(v,v)$ does not determine the
symmetric bilinear form, and the full pair $(\beta,\omega)$ must be
retained.
\end{example}

\begin{remark}
The role of the associative square in the present classification is
parallel to the role played by the Leibniz square $[x,x]$ in the
classification of Leibniz algebras whose subalgebras are ideals in
\cite{KurdachenkoSemkoSubbotin2017}. In a Poisson algebra the Lie
multiplication is alternating, so $[x,x]=0$ identically; the element
$x^2$ from the commutative associative multiplication becomes the
mechanism which forces the common one-dimensional central part. This
is the source of the condition $\beta(v,v)\neq0$ and of the close
appearance of bilinear forms in the two theories.
\end{remark}

\end{document}